\documentclass[12pt,leqno,draft]{amsart}
\usepackage{amsmath,color}
\def\BbbE{{{\rm I}  \! {\rm E}}}
\newcommand{\Var}{{\rm Var}}
\newcommand{\Cov}{{\rm Cov}}

\def \E{\BbbE}
\newcommand{\cal}{\mathcal}

\theoremstyle{plain}
\newtheorem{theorem}{Theorem}
\newtheorem{lemma}{Lemma}[section]
\newtheorem{proposition}[lemma]{Proposition}

\theoremstyle{definition}

\begin{document}
\bibliographystyle{plain}
\title[Estimating drift parameters in a fractional OU process]{Estimating drift parameters in a fractional Ornstein Uhlenbeck process with periodic mean}
\author[H. Dehling]{Herold Dehling}
\author[B. Franke]{Brice Franke}
\author[J.H.C. Woerner]{Jeannette H.C. Woerner}
\today
\address{Fakult\"at f\"ur Mathematik, Ruhr-Universit\"at Bochum,
44780 Bochum, Germany}
\email{herold.dehling@ruhr-uni-bochum.de}
\address{D\'{e}partement de Math\'{e}matique, 
Universit\'{e} de Bretagne Occidentale, 29200 Brest, France}
\email{brice.franke@univ-brest.fr}
\address{Fakult\"at f\"ur Mathematik, 
Technische Universit\"at Dortmund,
44221 Dortmund, Germany}
\email{jwoerner@mathematik.uni-dortmund.de}


\begin{abstract}
We construct a least squares estimator for the drift parameters of a fractional Ornstein Uhlenbeck process with periodic mean function and long range dependence. For this estimator we prove consistency and asymptotic normality. In contrast to the classical fractional Ornstein Uhlenbeck process without periodic mean function the rate of convergence is slower depending on the Hurst parameter $H$, namely $n^{1-H}$.
\end{abstract}
\maketitle

{\bf key words:} fractional Ornstein Uhlenbeck process, long range dependence, periodic mean function, least squares estimator
\section{Introduction}

The classical mean reverting Ornstein Uhlenbeck process is defined by the following stochastic differential equation
$$ dX_t=-\alpha X_tdt+\sigma dB_t ,$$
where $ \alpha $ and $ \sigma $ are positive real numbers and $ (B_t)_{t\geq0} $ is a standard Brownian motion.
Under Gaussian initial conditions the resulting process $ (X_t)_{t\geq0} $ is a Gaussian process.
This property and the strong mean reverting character of $ (X_t)_{t\geq0} $ lead to a long history in various applications. Slight modifications of the above SDE have been proposed to 
cover further features of observed data.  
For example we may add a deterministic drift term to the above stochastic differential equation;  i.e.:
$$ dX_t=(L(t)-\alpha X_t)dt+\sigma dB_t .$$
The resulting process is mean reverting around the time dependent center $ L(t) $. This may be used to model seasonalities 
or trends in the data.
However, the  process $ (X_t)_{t\geq0} $ is a semimartingale and thus the range of applications of the above models is limited 
to data which do not exhibit long range dependence. In situations with long range dependence but without mean reversion
fractional Brownian motion $ (B_t^H)_{t\geq0} $ is the classical model. Once the Hurst parameter $ H\in (0,1) $ fixed it is characterized as the unique continuous zero mean Gaussian process with
covariance 
$$ \E[B_t^HB_s^H]=\frac{1}{2}\left(|t|^{2H}+|s|^{2H}-|t-s|^{2H}\right) .$$ Hence we see that for $H=1/2$, we recover Brownian motion and for $H>1/2$ we have long range dependence. One possibility of a representation of fractional Brownian motion is the moving average representation in Mandelbrot and Van Ness \cite{MVN} of the form
\[ B_t^H= C \int_{-\infty}^{\infty} (a_1((t-s)_+^{H-1/2}-(-s)_+^{H-1/2})+a_2((t-s)_-^{H-1/2}-(-s)_-^{H-1/2}))dW_s,\]
where $x_+=\max(x,0)$, $x_-=\max(-x,0)$, $a_1,a_2\in  \mathbb{R}$, $H\in(0,1)$, $C$ is a normalizing constant and $W$ denotes a two-sided Brownian motion. This is defined by  $W_t=W^1_t$ if $t\geq 0$ and $W_t=-W^2_{-t}$ if $t<0$, where $W^1$ and $W^2$ denote independent copies of $W$.

In view of combining the mean reverting character of the Ornstein Uhlenbeck model with the long  memory property of the fractional Brownian motion we introduce for $H>1/2$ a generalized fractional Ornstein Uhlenbeck processes by the SDE
\begin{eqnarray}  \label{DGL}
   dX_t=(L(t)-\alpha X_t)dt+\sigma dB_t^H .
\end{eqnarray}

First we have to specify in which sense we interpret integrals with respect to fractional Brownian motion. Unless in the case of Brownian motion, where due to the semimartingale setting the It\^{o} integral is the canonical choice, this is not so clear for fractional Brownian motion with $H>1/2$. We will now interpret our integrals as divergence integral (cf.\ e.g. \cite{Nu}) since this type of integral allow us to derive a solution to (\ref{DGL}) in an analogous form as in the Brownian case (cf. \cite{DFK}). This would not be the case, if we interpret the integral as pathswise Riemann Stieltjes integral. Furthermore, in \cite{HuNu} it was shown that a least squares approach for the mean reverting parameter in a fractional Ornstein Uhlenbeck setting without periodic mean only leads to a consistent estimator when the integrals are interpreted as divergence integrals.

In this paper we want to study the particular class of periodic drift functions $ L(t)=\sum_{i=1}^p\mu_i\varphi_i(t) $,
where the functions $ \varphi_i(t) ; i=1,...,p $ are bounded and periodic with the same period $ \nu $ and the real numbers 
$ \mu_i;i=1,...,p $ are unknown parameters. Those parameters together with the positive real number $ \alpha $ will be used to fit
the model to the data. The parameter $ \sigma>0 $ will be assumed known, since for continuous observations it may be read of the sample paths immediately. 
Thus we have to estimate $ \theta:=(\mu_1,...,\mu_p,\alpha)^t $ from an $ p+1 $-dimensional parameter space $ \Theta:=\mathbb{R}^p\times\mathbb{R}^+ $.
Concerning the sampling scheme we assume that we observed the full path of the process $ (X_t)_{t\geq0} $ on a time interval 
$ T $ which is a multiple of the time period $ \nu $; i.e. $ T=n\nu $ for some $ n\in\mathbb{N} $ and let $n\to\infty$.
In the following the Lebesgue measure on the real line will be denoted by $ \ell $. 
Without loss of generality we assume that the functions $ \varphi_i;i=1,...,p $ are orthonormal in $ L^2([0,\nu],\nu^{-1}\ell) $; i.e.
$$ \int_0^\nu \varphi_i(t)\varphi_j(t)dt=\delta_{ij}\nu .$$
Moreover, we will assume that the functions $ \varphi_i;i=1,...,p $ are bounded by a constant $ C>0 $. 
The resulting SDE then takes the form
\begin{eqnarray*} 
      dX_t=\left(\sum_{i=1}^p\mu_i\varphi_i(t)-\alpha X_t\right)dt+\sigma dB_t^H 
\end{eqnarray*}
 with initial condition $ X_0=\xi_0 $, where $ \xi_0 $ is a random variable independent of the fractional Brownian motion $ (B^H_t)_{t\in\mathbb{R}} $.
 
 The outline of the paper is the following. First we provide some properties of the model taking into account the special features of the periodic mean function. In section 3 we derive our estimator and in section 4 and 5 we prove consistency and asymptotic normality, respectively.

\section{Some preliminary facts on the model}
First of all we need some properties of our model which rely on the periodicity of the mean function.
Now we derive an explicit solution for equation (\ref{DGL}).

\begin{proposition}
The following stochastic process $ (X_t)_{t\geq0} $ given by 
$$  X_t=e^{-\alpha t}\left(\xi_0+\int_0^t e^{\alpha s}L(s)ds-\sigma\int_0^te^{\alpha s}dB^H_s\right); \ t\geq0 $$
is the unique almost surely continuous solution of equation (\ref{DGL}) with initial condition $ X_0=\xi_0 $.
\end{proposition}
\begin{proof}
This follows analogously to the Brownian case in \cite{DFK} from the representation of the fractional Ornstein Uhlenbeck process with initial condition $ \xi_0 $ in \cite{CKM}, since the It\^{o} formula for the divergence integral of fractional Brownian motion only differs in the term of the second derivatives, which do not contribute in the Ornstein-Uhlenbeck setting, cf. \cite{Nu}.
\end{proof}
In the following we need a stationary solution of equation (\ref{DGL}) in order to prove a necessary ergodic theorem. The stationary solution is given in an analogous way as moving average as in the classical Ornstein Uhlenbeck process.

\begin{proposition}
The following stochastic process $ (\tilde{X}_t)_{t\geq0} $ given by
$$ \tilde{X}_t:=e^{-\alpha t}\left(\int_{-\infty}^te^{\alpha s}L(s)ds+\sigma\int_{-\infty}^te^{\alpha s}dB^H_s\right) $$
is an almost surely continuous solution of equation (\ref{DGL}).
\end{proposition}
\begin{proof}
This follows from the representation of the stationary fractional Ornstein Uhlenbeck process in \cite{CKM}. Note that we do not have any problem with the integral starting from $-\infty$ since the moving average representation defines the fractional Brownian motion on the whole real line.
\end{proof}

Next we show that for large $t$ the difference between the two representations tends to zero.

\begin{proposition} \label{Approx}
As $ t\rightarrow\infty $ we obtain  almost surely that $ |X_t-\tilde{X}_t|\rightarrow 0$.
\end{proposition}
\begin{proof}
The explicit representations of $ X_t $ and $ \tilde{X}_t $ yield
\begin{eqnarray*}
 |X_t-\tilde{X}_t| & \leq & e^{-\alpha t}|\xi_0|+e^{-\alpha t}\left|\int_{-\infty}^0 e^{\alpha s}L(s)ds\right| 
                                          + \sigma e^{-\alpha t}\left|\int_{-\infty}^0e^{\alpha s}dB^H_s\right| ,
\end{eqnarray*}
where the right hand side tends to zero almost surely as $t\to\infty$.
\end{proof}
From this solution we may now construct a stationary and ergodic sequence of random variables which we need later for our limit theorems.
\begin{proposition} \label{Ergodisch}
Assume that $L$ is periodic with period 1, then the sequence of $ C[0,1] $-valued random variables $$ W_k(s):=\tilde{X}_{k-1+s} , 0\leq s\leq1, k\in\mathbb{N} $$
is stationary and ergodic.
\end{proposition}
\begin{proof}
Since $L$ is periodic, the function $$ \tilde{h}(t):=e^{-\alpha t}\int_{-\infty}^te^{\alpha s}L(s)ds $$ is also periodic on $ \mathbb{R} $.
We have for any $ t\in[0,1] $ that
\begin{eqnarray*}
  W_k(t) &=& e^{-\alpha (k-1+t)}\int_{-\infty}^{k-1+t}e^{\alpha s}L(s)ds +\sigma e^{-\alpha (k-1+t)}\int_{-\infty}^{k-1+t}e^{\alpha s}dB_s^H \\
            &=& \tilde{h}(t) + \sigma e^{-\alpha(k-1+t)}\int_{k-1}^{k-1+t}e^{\alpha s}dB_s^H +\sigma \sum_{l=-\infty}^{k-1}e^{-\alpha(k-1+t)}\int_{l-1}^le^{\alpha s}dB_s^H \\
    &=& \tilde{h}(t) + \sigma e^{-\alpha t}\int_0^t e^{\alpha s} dB^H_{s+k-1} +\sigma \sum_{l=-\infty}^{k-1}e^{-\alpha(k-l+t)}\int_0^1e^{\alpha s}dB_{s+l-1}^H \\
    &=&  \tilde{h}(t) + \sigma e^{-\alpha t}\int_0^t e^{\alpha s} dB^H_{s+k-1} +\sigma \sum_{j=-\infty}^0e^{-\alpha(t+1-j)}\int_0^1e^{\alpha s}dB_{s+j+k-2}^H .
\end{eqnarray*}
Thus, we have the almost sure representation
\begin{eqnarray*}
  W_k(\cdot) &=& \tilde{h}(\cdot)+F_0(Y_k)+\sum_{j=-\infty}^0 e^{\alpha(j-1)}F(Y_{j+k-1})
\end{eqnarray*}
with the functionals
$$ F_0:C[0,1]\rightarrow C[0,1];\omega\mapsto\left(t\mapsto\sigma e^{-\alpha t}\int_0^t e^{\alpha s}d\omega(s)\right) ,$$
$$ F:C[0,1]\rightarrow \mathbb C[0,1];\omega\mapsto\sigma e^{-\alpha t}\int_0^1 e^{\alpha s}d\omega(s) $$
and the $ C[0,1] $-valued random variable
$$ Y_l:=\left[s\mapsto B^H_{s+l-1}-B^H_{l-1}; 0\leq s\leq 1\right] .$$
The sequence of Gaussian random variables $ (Y_l)_{l\in\mathbb{Z}} $ is stationary and ergodic. This implies that the sequence of $ C[0,1] $-valued random variables 
$ (W_k)_{k\in\mathbb{N}} $ is stationary and ergodic.
\end{proof}

\section{The estimator and its motivation}

In this section we want to motivate a particular estimator $ \hat{\theta} $ for the parameter $ \theta $ through a least squares approach. 
The same kind of approach was used in \cite{FK} to derive an estimator for drift parameters in a L\'evy driven stochastic differential equations.
First, we analyze the more general estimation problem of a  $ p+1 $-dimensional  parameter vector $ \theta=(\theta_1,...,\theta_{p+1}) $ in the stochastic differential equation 
$$  dX_t=\theta f(t,X_t)dt+\sigma dB_t^H ,$$
where $ f(t,x)=(f_1(t,x),...,f_{p+1}(t,x))^t $ with suitable real valued functions $ f_i(t,x); 1\leq i\leq p $. 
A discretization of the above equation on the time interval $ [0,T] $ yields for $ \Delta t:=T/N $ and $ i=1,...,N $
\begin{eqnarray} \label{DiskDGL}  
X_{(i+1)\Delta t}-X_{i\Delta t}=\sum_{j=1}^{p+1}f_j(i\Delta t,X_{i\Delta t})\theta_j\Delta t +\sigma\left(B^H_{(i+1)\Delta t}-B^H_{i\Delta t}\right) .
\end{eqnarray}
The system of equations (\ref{DiskDGL}) has some resemblance to classical linear models with dependent noise.
In accordance to the classical least squares approach for estimation in linear models we can try to minimize the least squares functional
\begin{eqnarray} \label{KQFunktional} 
 {\cal G}:(\theta_1,...,\theta_{p+1})\mapsto\sum_{i=1}^N\left(X_{(i+1)\Delta t}-X_{i\Delta t}-\sum_{j=1}^{p+1}f_j(i\Delta t,X_{i\Delta t})\theta_j\Delta t\right)^2 .
\end{eqnarray}
It was discussed in \cite{FK} that the minimizer $ \tilde{\theta}_{T,\Delta t} $ of the above functional has the form
$$ \tilde{\theta}_{T,\Delta t}=Q_{T,\Delta t}^{-1}P_{T,\Delta t} ,$$
with
\[  Q_{T,\Delta t}=\left(\begin{array}{ccc}   
\sum_{i=0}^Nf_1(i\Delta t,X_{i\Delta t})f_1(i,\Delta t,X_{i\Delta t})\Delta t & \! \dots \! &  \sum_{i=0}^Nf_1(i\Delta t,X_{i\Delta t})f_p(i,\Delta t,X_{i\Delta t})\Delta t \\
 \vdots & \! \!   & \vdots \\
\sum_{i=0}^Nf_p(i\Delta t,X_{i\Delta t})f_1(i,\Delta t,X_{i\Delta t})\Delta t & \!\dots \! & \sum_{i=0}^Nf_p(i\Delta t,X_{i\Delta t})f_p(i,\Delta t,X_{i\Delta t})\Delta t 
\end{array}\right)
\]
and
\[ P_{T,\Delta t}:=\left(\sum_{i=1}^Nf_1(i\Delta t,X_{i\Delta t})(X_{(i+1)\Delta t}-X_{i\Delta t}),...,\sum_{i=1}^Nf_{p+1}(i\Delta t,X_{i\Delta t})(X_{(i+1)\Delta t}-X_{i\Delta t})\right)^t.
\]
This motivates the continuous time estimator $ \hat{\theta}_T=Q_T^{-1}P_T $ with
\[  Q_T=\left(\begin{array}{ccc}   
 \int_0^T f_1(t,X_t)f_1(t,X_t)dt  &  \dots  & \int_0^T f_1(t,X_t)f_{p+1}(t,X_t)dt \\
 \vdots & \! \!   & \vdots \\
\int_0^T f_{p+1}(t,X_t)f_1(t,X_t)dt & \dots & \int_0^T f_{p+1}(t,X_t)f_{p+1}(t,X_t)dt
\end{array}\right)
\]
and
\[ P_{T}:=\left( \int_0^T f_1(t,X_t)dX_t ,..., \int_0^T f_p(t,X_t)dX_t \right)^t .
\]
In the special case of the fractional Ornstein Uhlenbeck process we have $ \theta=(\mu_1,...,\mu_p,\alpha) $ and
$ f(t,x):=\left(\varphi_1,...,\varphi_p,-x\right)^t  $. This yields for $ T=n\nu $ the estimator
\begin{equation} \label{thetahat} \hat{\theta}_n:=Q_n^{-1}P_n \end{equation}
with 
$$ P_n:=\left(\int_0^{n\nu}\varphi_1(t)dX_t,...,\int_0^{n\nu}\varphi_p(t)dX_t,-\int_0^{n\nu}X_tdX_t\right)^t  $$
and 
\[   Q_n:=\left(\begin{array}{cc} G_n & -a_n \\ -a_n^t & b_n \end{array}\right) , \]
where
\[  G_n:= \left( \begin{array}{ccc} 
     \int_0^{n\nu}\varphi_1(t)\varphi_1(t)dt & \dots &   \int_0^{n\nu}\varphi_1(t)\varphi_p(t)dt \\
    \vdots & & \vdots \\
      \int_0^{n\nu}\varphi_p(t)\varphi_1(t)dt & \dots &   \int_0^{n\nu}\varphi_p(t)\varphi_p(t)dt \\
    \end{array}\right)=n\nu I_p  ,\]
$$   a_n^t:=\left(\int_0^{n\nu}\varphi_1(t)X_tdt,...,\int_0^{n\nu}\varphi_p(t)X_tdt\right)  $$
and
$$   b_n:=\int_0^{n\nu}X_t^2dt .$$

Note that $I_p$ denotes the $p$-dimensional unit matrix.

First we deduce an explicit representation of the estimator $ \hat{\theta}_n $. For simplicity from now on we set $\nu=1$.
\begin{proposition} \label{ThetaVariation}
We have $ \hat{\theta}_n=\theta+\sigma Q_n^{-1}R_n $ with 
$$  R_n:=\left(\int_0^{n}\varphi_1(t)dB_t^H,...,\int_0^{n}\varphi_p(t)dB_t^H,-\int_0^{n}X_tdB_t^H\right)^t .$$
\end{proposition}
\begin{proof}
Using equation (\ref{DGL}), we obtain for the $i$-th component ($1\leq i\leq p$) of the vector $ P_n $ the representation
$$     \int_0^{n} \varphi_i(t)dX_t=\sum_{j=1}^p\mu_j\int_0^{n}\varphi_i(t)\varphi_j(t)dt-\alpha\int_0^{n}\varphi_i(t)X_tdt+\sigma\int_0^{n}\varphi_i(t)dB_t^H  $$
and for the $p+1$-th component
$$   \int_0^{n}X_tdX_t=\sum_{j=1}^p\mu_j\int_0^{n}X_t\varphi_j(t)dt-\alpha\int_0^{n}X_t^2dt+\sigma\int_0^{n}X_tdB^H_t . $$
This yields $ P_n=Q_n\theta+\sigma R_n $ from which together with $ \hat{\theta}=Q_n^{-1}P_n $ proves the claim of the proposition.
\end{proof}

Furthermore, we can compute the matrix $ Q_n^{-1} $ explicitly.

\begin{proposition} \label{KonkretMatrix}
We obtain
\[  Q_n^{-1} = \frac{1}{n}\left(\begin{array}{cc} I_p+\gamma_n\Lambda_n\Lambda_n^t & -\gamma_n\Lambda_n \\
                                                                                   -\gamma_n\Lambda_n^t  & \gamma_n      \end{array}\right) .
\]
with
$$  \Lambda_n=\left(\Lambda_{n,1},...,\Lambda_{n,p}\right)^t:=\left(\frac{1}{n}\int_0^n\varphi_1(t)X_tdt,...,\frac{1}{n}\int_0^n\varphi_p(t)X_tdt\right)^t   $$
and
$$ \gamma_n :=\left(\frac{1}{n}\int_0^nX_t^2dt-\sum_{i=1}^p\Lambda_{n,i}^2\right)^{-1}   .$$
\end{proposition}
\begin{proof}
See \cite{DFK} for the proof.
\end{proof}

Note that, since we will show in Proposition \ref{MatrixKonvergenz} that the limit of $nQ_n^{-1}$ is well defined, especially $\lim_{n\to\infty} \gamma_n=\gamma>0$, this implies that for large enough $n$ also $nQ_n^{-1}$ is well defined almost surely.

\section{The consistency of the estimator}

In this section we prove that the estimator $ \hat{\theta}_n  $ is consistent using the representation of  Proposition \ref{ThetaVariation}.

\begin{theorem} \label{consistent}
For $H\in(1/2, 3/4)$ $ \hat{\theta}_n $ defined in (\ref{thetahat}) converges in probability to $ \theta $ as $ n\rightarrow\infty $.
\end{theorem}
\begin{proof}
Since by Proposition \ref{ThetaVariation} we have $ \hat{\theta}_n=\theta+\sigma nQ_n^{-1}\frac{1}{n}R_n $ the statement of the theorem follows directly by the following two Propositions, Proposition \ref{MatrixKonvergenz} and Proposition  \ref{VektorKonvergenz}.
\end{proof}
\begin{proposition} \label{VektorKonvergenz}
For $H\in(1/2, 3/4)$ the sequence $ n^{-H}R_n $ is bounded in $ L^2$.
\end{proposition}
\begin{proof}
Assume that $\sup_t|\varphi_i(t)|\leq C<\infty$, then by the isometry for fractional Brownian motion the variance of the first $ p $ components is
\begin{eqnarray*}
 \Var\left(n^{-H}\int_0^n\varphi_i(t)dB_t^H\right)&=&\E\left[\left(n^{-H}\int_0^n\varphi_i(t)dB_t^H\right)^2\right]  \\
       &=& H(2H-1)n^{-2H}\int_0^n\int_0^n \varphi_i(u)\varphi_i(v)|u-v|^{2H-2}dudv \\
 &\leq&  2C^2 H(2H-1)n^{-2H}\int_0^n\int_0^v(v-u)^{2H-2}dudv =C^2.
\end{eqnarray*}
In order to compute the variance of the last component, we write the solution of equation (\ref{DGL}) in the form
$$ X_t=e^{-\alpha t}\xi_0+h(t)+Z_t $$
with
$$ h(t):=e^{-\alpha t}\sum_{i=1}^p\mu_i\int_0^t e^{\alpha s}\varphi_i(s)ds   $$
and 
$$ Z(t):= \sigma e^{-\alpha t}\int_0^t e^{\alpha s}dB_s^H.  $$
The variance of the last component then is               
\begin{eqnarray*}
  \Var\left(n^{-H}\int_0^nX_tdB_t^H\right) &=& n^{-2H}\E\left[\left(\int_0^nX_tdB_t^H\right)^2\right]\\
     &=& n^{-2H} \E\left[\left(\int_0^n\left(e^{-\alpha t}\xi_0+h(t)+Z_t\right) dB_t^H\right)^2\right] \\
     &=& E_1+E_2+E_3+E_4+E_5+E_6,
\end{eqnarray*}
with
\begin{eqnarray*}
 E_1 &:=& n^{-2H}\E\left[\left(\int_0^ne^{-\alpha t}\xi_0dB_t^H\right)^2\right]=\E\left[\xi_0^2\right]\int_0^1\int_0^1e^{-\alpha n(u+v)}|u-v|^{2H-2}dudv ,\\
  E_2 &:=&  2n^{-2H}\E\left[\int_0^ne^{-\alpha t}\xi_0dB_t^H\int_0^nh(t)dB_t^H\right] \\
         &=& 2\E[\xi_0]\int_0^1\int_0^1e^{-\alpha nu}h(nu)|u-v|^{2H-2}dudv ,\\
  E_3 &:=& 2n^{-2H}\E\left[\int_0^ne^{-\alpha r}\xi_0dB_r^H\int_0^nZ_tdB_t^H\right]\\
 E_4 &:=& n^{-2H}\E\left[\left(\int_0^nh(t)dB_t^H\right)^2\right] =\int_0^1\int_0^1h(nu)h(nv)|u-v|^{2H-2}dudv ,\\
  E_5 &:=& 2n^{-2H}\E\left[\int_0^nh(r)dB_r^H\int_0^n Z(t)dB_t^H\right] \end{eqnarray*}
and
$$  E_6:=n^{-2H}\E\left[\left(\int_0^nZ_tdB_t^H\right)^2\right]=n^{-2H}\sigma^2\E\left[\left(\int_0^n\int_0^te^{-\alpha(t-s)}dB_s^HdB_t^H\right)^2\right] .$$
We now have to discuss the boundedness of each of those terms separately. The term $ E_1 $ is obviously bounded as $ n\rightarrow\infty $.
Since the function $ t\mapsto h(t) $ is bounded, it follows also that the terms $ E_2 $ and $ E_4 $ stay bounded. $ E_3 $ and $ E_5 $ are zero by the isometry property for multiple Wiener integrals of different order (cf.\ e.g. \cite{NP}).
 $ E_6 $ also stays bounded as $n\to\infty$, since we know by \cite{HuNu} that $\frac{1}{n}\E\left[\left(\int_0^nZ_tdB_t^H\right)^2\right]$ is convergent for $H\in(1/2,3/4)$.
\end{proof}

\begin{proposition}\label{MatrixKonvergenz}
As $ n\rightarrow\infty $ we obtain that $ nQ_n^{-1} $ converges almost surely to 
\[ C:=\left(\begin{array}{cc}   I_p+\gamma\Lambda\Lambda^t & -\gamma\Lambda \\
                                                -\gamma\Lambda^t & \gamma \end{array}\right),\] 
where
$$  \Lambda=(\Lambda_1,...,\Lambda_p)^t:=\left( \int_0^1\varphi_1(t)\tilde{h}(t)dt,..., \int_0^1\varphi_p(t)\tilde{h}(t)dt \right)^t $$
and
$$ \gamma:=\left(\int_0^t\tilde{h}^2(t)dt+\sigma^2\alpha^{-2H}H\Gamma(2H)-\sum_{i=1}^p\Lambda_i^2\right)^{-1}, $$
with $\tilde{h}(t):=e^{-\alpha t}\sum_{i=1}^p\mu_i\int_{-\infty}^te^{\alpha s}\varphi_i(s)ds $.

\end{proposition}
\begin{proof}
We use the following notation $$ \tilde{X}_t=\tilde{h}(t)+\tilde{Z}_t $$ with
$$ \tilde{h}(t):=e^{-\alpha t}\sum_{i=1}^p\mu_i\int_{-\infty}^te^{\alpha s}\varphi_i(s)ds $$
being periodic with period 1 and
$$  \tilde{Z}_t:=\sigma e^{-\alpha t}\int_{-\infty}^te^{\alpha s}dB_s^H .  $$
We investigate the limit behaviour of the different entries in $ n^{-1}Q_n $ separately.

The ergodic theorem, Proposition \ref{Ergodisch} and Proposition \ref{Approx} yield
\begin{eqnarray*}
 \lim_{n\rightarrow\infty}\Lambda_{n,i} &=& \lim_{n\rightarrow n}\frac{1}{n}\int_0^n X_t\varphi_i(t)dt
     = \lim_{n\rightarrow n}\frac{1}{n}\int_0^n \tilde{X}_t\varphi_i(t)dt 
  =  \lim_{n\rightarrow n}\frac{1}{n}\sum_{k=1}^n\int_{k-1}^k\tilde{X}_t\varphi_i(t)dt \\
&=& \E\left[\int_0^1\tilde{X}_t\varphi_i(t)dt\right] = \int_0^1 \tilde{h}(t)\varphi_i(t)dt+\int_0^1\E\left[\int_0^t\varphi_i(t)e^{-\alpha (t-s)}dB_s^H\right]dt,
\end{eqnarray*} 
where the second term in the last expression is zero by the properties of Wiener integrals.

Note that the sequence of $ C[0,1] $ valued random variables $ \left[s\mapsto\tilde{Z}_{k+s}\right]; k\in\mathbb{Z} $ is stationary and ergodic, hence by the ergodic theorem we have as $ n\rightarrow\infty $
$$  \frac{1}{n}\int_0^n\tilde{Z}_tdt =\frac{1}{n}\sum_{k=1}^n\int_{k-1}^k\tilde{Z}_t \longrightarrow \E\left[\int_0^1\tilde{Z}_tdt\right]=\E[\tilde{Z}_0]=0. $$
Moreover, we have 
\begin{eqnarray*}
  \left|\frac{1}{n}\int_0^n\left(Z_t-\tilde{Z}_t\right)dt\right| \leq \frac{\sigma}{n}\int_0^ne^{-\alpha t}\left|\int_{-\infty}^0e^{\alpha s}dB_s^H\right|dt\longrightarrow 0 \ \ \ 
\mbox{as $ n\rightarrow\infty $}.
\end{eqnarray*}
Those two facts now imply that
$$  \lim_{n\rightarrow\infty}\left|\frac{1}{n}\int_0^nZ_tdt \right|=0   . $$
It then follows from the boundedness of the function $ t\mapsto h(t) $ that
$$   \limsup_{n\rightarrow\infty}\left|\frac{1}{n}\int_0^n X_tdt\right|= \limsup_{n\rightarrow\infty}\left|\frac{1}{n}\int_0^n \left(e^{-\alpha t}\xi_0+h(t)+Z_t\right)dt \right|<\infty $$
From Proposition \ref{Approx} we also have
$$  \limsup_{n\rightarrow\infty}\left|\frac{1}{n}\int_0^n \tilde{X}_tdt\right|<\infty .$$
Those two inequalities together with Proposition \ref{Approx} yield
$$  \left|\frac{1}{n}\int_0^n\tilde{X}_t^2dt-\frac{1}{n}\int_0^nX_t^2dt\right|=\left|\frac{1}{n}\int_0^n\left(\tilde{X}_t+X_t\right)\left(\tilde{X}_t-X_t\right)dt\right|\longrightarrow0
   \ \ \ \mbox{as $ n\rightarrow\infty $}  .$$
Now using the ergodic theorem for $\tilde{h}$ and \cite{HuNu} for $\tilde{Z}$ we obtain 
\begin{eqnarray*}
    \lim_{n\rightarrow\infty}\frac{1}{n}\int_0^nX_t^2dt &=& \lim_{n\rightarrow\infty}\frac{1}{n}\int_0^n\tilde{X}_t^2dt\\&=&\int_0^1\tilde{h}^2(t)dt+\sigma^2\alpha^{-2H}H\Gamma(2H) ,
\end{eqnarray*}
noting that the mixed term is zero due to the properties of Wiener integrals.

Thus the general expression for $ \gamma_n $ in Proposition \ref{KonkretMatrix} yields
\begin{eqnarray*} 
   \lim_{n\rightarrow\infty}\gamma_n &=& \lim _{n\rightarrow\infty}\left(\frac{1}{n}\int_0^nX_t^2dt-\sum_{i=1}^p\Lambda_{n,i}^2\right)^{-1} \\
   &=& \left(\int_0^t\tilde{h}^2(t)dt+\sigma^2\alpha^{-2H}H\Gamma(2H)-\sum_{i=1}^p\Lambda_i^2\right)^{-1}.
\end{eqnarray*}
Since the functions $ \varphi_i; i=1,...,p $ are orthonormal in $ L^2[0,1] $ we can use the Bessel inequality 
$$ \sum_{i=1}^p\Lambda_i^2 =\sum_{i=1}^p\left(\int_0^1 \varphi_i(t)\tilde{h}(t)dt\right)^2 \leq \int_0^1 \tilde{h}^2(t) dt $$
to see that the above limit is well defined and finite, which completes our proof.
\end{proof}

\section{The asymptotic normality of the estimator}
In this section we prove asymptotic normality of our estimator which may be reduced to a limit theorem for dependent normally distributed random  variables.

\begin{theorem}
For $H\in(1/2,3/4)$ we obtain for $\hat{\vartheta}_n$ defined by (\ref{thetahat})
$$ n^{1-H} (\hat{\vartheta}_n-\vartheta)\stackrel{\mathcal D}{\longrightarrow} {\mathcal N}(0,\sigma^2 C\Sigma_0 C) $$ 
with
\[  \Sigma_0 :=\left(\begin{array}{cc} \bar{G} & -\bar{a} \\ -\bar{a}^t & \bar{b} \end{array}\right) , \]
where
\[  \bar{G}:= \left( \begin{array}{ccc} 
     \alpha_H\int_0^1\int_0^1\varphi_1(s)\varphi_1(t)|t-s|^{2H-2}dsdt & \dots & \alpha_H\int_0^1\int_0^1\varphi_1(s)\varphi_p(t)|t-s|^{2H-2}dsdt  \\
    \vdots & & \vdots \\
      \alpha_H\int_0^1\int_0^1\varphi_p(s)\varphi_1(t)|t-s|^{2H-2}dsdt & \dots &   \alpha_H\int_0^1\int_0^1\varphi_p(s)\varphi_p(t)|t-s|^{2H-2}dsdt \\
    \end{array}\right)  ,\]
$$   \bar{a}^t:=\left(\alpha_H\int_0^1\int_0^1\varphi_1(s)\tilde{h}(t)|t-s|^{2H-2}dsdt,...,\alpha_H\int_0^1\int_0^1\varphi_p(s)\tilde{h}(t)|t-s|^{2H-2}dsdt\right),  $$
$$   \bar{b}:=\alpha_H\int_0^1\int_0^1\tilde{h}(s)\tilde{h}(t)|t-s|^{2H-2}dsdt, $$
$$\alpha_H=H(2H-1),$$
$$\tilde{h}(t):=e^{-\alpha t}\sum_{i=1}^p\mu_i\int_{-\infty}^te^{\alpha s}\varphi_i(s)ds $$
and $C$ is defined in Proposition (\ref{MatrixKonvergenz}).
\end{theorem}
Let us discuss the differences to the Brownian case (cf. \cite{DFK}) and the fractional Ornstein Uhlenbeck case without periodic mean function (cf. \cite{HuNu}) before proceeding with the proof. The rate of convergence $n^{1-H}$ is slower than in the Brownian case. Furthermore, it is also slower than the rate $n^{1/2}$ for the mean reverting parameter in a fractional Ornstein Uhlenbeck setting with $L=0$. This is due to the special structure of our drift coefficient, which in our setting also dominates the component of $\alpha$ leading to a slower rate even for $\alpha$ and a different entry in the covariance matrix. Furthermore, unless in the Brownian case $\Sigma_0\neq C^{-1}$. This is due to the isometry formula for fractional Brownian motion with $H>1/2$, which is not simply derived from the scalar product in $L^2$, but from the scalar product in a larger Hilbert space ${\cal{H}}$. Namely for a fixed time interval $[0,T]$  ${\cal{H}}$ is defined as the closure of the set of real valued step functions on $[0,T]$ with respect to the scalar product $<1_{[0,t]}, 1_{[0,s]}>_{{\cal{H}}}=\E(B_t^HB_s^H)$.
\begin{proof}
By the representation
\[\hat{\vartheta}_n-\vartheta=\sigma Q_n^{-1}R_n\]
and the almost sure convergence of $nQ_n^{-1}\to C$ it is sufficient to prove that as $n\to\infty$
$$ \left(n^{-H}\int_0^n \varphi_1(t)dB^H_t,...,n^{-H}\int_0^n \varphi_p(t)dB^H_t, - n^{-H}\int_0^n  X_t dB^H_t\right)^t \stackrel{\mathcal D}{\longrightarrow} {\mathcal N}(0,\Sigma_0). $$
As in the proof of Proposition \ref{MatrixKonvergenz} we may replace $X_t$ by $\tilde{X}_t$ with the representation $\tilde{X}_t=\tilde{Z}_t+\tilde{h}(t)$ and may deduce that $\tilde{Z}_t$ does not contribute to the covariance matrix. Namely the contributions to the off-diagonal elements in $\bar{a}$ and the mixed term of $\bar{b}$ are zero by the isometry formula for multiple Wiener integrals of different order (cf.\ e.g. \cite{NP}). Furthermore, $\Var(n^{-H}\int_0^n\tilde{Z}_t dB_t^H)\to 0$ as $n\to\infty$, since we know by \cite{HuNu} that $\frac{1}{n}\Var(\int_0^n\tilde{Z}_t dB_t^H)$ is convergent and $2H>1$ for $1/2<H<3/4$.

Hence it is sufficient to show that for the $1$-periodic functions $\varphi_i$ $(1\leq i \leq p)$ and $\tilde{h}$ as $n\to\infty$
$$ \left(n^{-H}\int_0^n \varphi_1(t)dB^H_t,...,n^{-H}\int_0^n \varphi_p(t)dB^H_t, - n^{-H}\int_0^n  \tilde{h}(t) dB^H_t\right)^t \stackrel{\mathcal D}{\longrightarrow} {\mathcal N}(0,\Sigma_0). $$
This is an immediate consequence of the following Proposition \ref{Limit}.
\end{proof}
\begin{proposition}\label{Limit}
Let $f_k$ $(1\leq k\leq m)$ be periodic real valued functions with period 1, then for $H>1/2$ and $ n\rightarrow\infty $ 
\begin{eqnarray*} \lefteqn{\left(n^{-H}\int_0^n f_1(t)dB^H_t,...,n^{-H}\int_0^n f_m(t)dB^H_t\right)^t}\\ &&\stackrel{\mathcal D}{\longrightarrow} {\mathcal N}\left(0,H(2H-1)\left(\int_0^1\int_0^1f_i(t)f_j(s)|t-s|^{2H-2}dsdt\right)_{1\leq i,j \leq m} \right). 
\end{eqnarray*} 
\end{proposition}
\begin{proof}
Since $f_k$ is periodic with period 1, we may write for $1\leq k\leq m$
$$ n^{-H}\int_0^n f_k(t)dB^H_t=n^{-H}\sum_{i=1}^n\int_{i-1}^i f_k(t)dB^H_t=n^{-H}\sum_{i=1}^n Y_i^k$$
with
$$ Y_i^k\sim {\mathcal N}\left(0, H(2H-1)\int_0^1\int_0^1f_k(t)f_k(s)|t-s|^{2H-2}dsdt\right)$$
and
\begin{eqnarray*}
\Cov(Y_i^k, Y_j^l)&=&\rho_H(|i-j|) H(2H-1)\int_0^1\int_0^1f_k(t)f_l(s)|t-s|^{2H-2}dsdt \\
&\sim& n^{2H-2} H^2(2H-1)^2\int_0^1\int_0^1f_k(t)f_l(s)|t-s|^{2H-2}dsdt
\end{eqnarray*}
for $1\leq i,j \leq n$ and $1\leq k,l \leq m$, since
$$\rho_H(n)=\frac{1}{2}((n+1)^{2H}+(n-1)^{2H}-2n^{2H}).$$
Hence the sequences $(Y_i^k)_i$ satisfy the conditions of Theorem 7.2.11 in \cite{ST}, which implies as $n\to\infty$
$$n^{-H}\sum_{i=1}^n Y_i^k\stackrel{\mathcal D}{\longrightarrow} {\mathcal N}\left(0,H(2H-1)\int_0^1\int_0^1f_k(t)f_k(s)|t-s|^{2H-2}dsdt\right).$$
Finally the Cramer-Wold device together with a similar argument for the covariance terms lead to the desired result.
\end{proof}
{\bf Acknowledgement:} The financial support of the DFG (German Science Foundation) SFB 823: Statistical modeling of nonlinear dynamic processes (projects C3 and C5) is gratefully acknowledged.

\end{document}